\documentclass[reqno,11pt]{amsart}
\usepackage[utf8]{inputenc}
\usepackage{mathtools}
\usepackage{amsmath}
\usepackage{amssymb}
\usepackage{bbm}
\usepackage{tikz-network}
\usepackage{enumitem}
\usepackage{float}
\usepackage{hyperref}

\newtheoremstyle{normal}
  {1em plus .2em minus .1em} 
  {1em plus .2em minus .1em} 
  {\normalfont} 
  {} 
  {\bfseries} 
  {} 
  {.5em} 
  {} 

\newtheorem{theorem}{Theorem}
\newtheorem{definition}{Definition}
\newtheorem{lemma}{Lemma}
\newtheorem{corollary}{Corollary}

\theoremstyle{normal}

\newtheorem{assumption}{Assumption}
\newtheorem{remark}{Remark}

\newcommand{\G}{\overrightarrow{\mathbb{G}}}
\newcommand{\E}{\overrightarrow{E}}
\newcommand{\F}[1]{\mathcal{F}_{#1}}
\newcommand{\C}[2]{C_{#1}^{(#2)}}
\newcommand{\U}[2]{U_{#1}^{(#2)}}
\newcommand{\us}[2]{u_{#1}^{(#2)}}
\newcommand{\Z}[2]{Z_{#1}^{(#2)}}

\newcommand{\z}[2]{z_{#1}^{(#2)}}
\newcommand{\N}[2]{\mathcal N_{#1}^{(#2)}}
\newcommand{\bb}[1]{\mathbbm{#1}}

\newcommand{\Q}{\mathbb{P}}
\newcommand{\Qr}{\mathbb{Q}}
\newcommand{\rs}[2]{#1\rightsquigarrow#2}
\newcommand{\rt}[2]{#1\rightarrow#2}
\newcommand{\e}{\overrightarrow{e}}

\title{Interacting Urn Schemes on Finite Ancestral Directed Acyclic Graphs}
\author{Antar Bandyopadhyay and Deborshi Das}
\address{Theoretical Statistics and Mathematics Unit\\
Indian Statistical Institute, Delhi Centre\\
7 S. J. S. Sansanwal Marg\\
New Delhi 110016, INDIA}
\email{antar@isid.ac.in and deborshidas6@gmail.com}

\begin{document}
\begin{abstract}
 We study interacting finite-color urn schemes on directed acyclic graphs, allowing the graph to be infinite. Each urn evolves through reinforcements driven by colors drawn from its in-neighboring urns via edge-dependent reinforcement matrices. Assuming that every vertex has only finitely many ancestors, we prove almost sure convergence of urn proportions and show that the limiting configuration is determined by vertices with no ancestors or self-loops. Under additional balance and irreducibility assumptions on reinforcement matrices, we also obtain second-order asymptotic results in all regimes of appropriately defined parameters.


\end{abstract}
\maketitle
\pagestyle{plain}

\section{Introduction}
The classical Pólya urn was introduced by Eggenberger and Pólya \cite{eggenberger1923statistik}, followed by Pólya's influential interpretation of the model as a mechanism for contagion and epidemic spread \cite{polya1930quelques}. Since then, Pólya urn models and their generalizations have been extensively studied as fundamental examples of reinforced stochastic processes, with applications in statistics, machine learning, population genetics, network growth, and social dynamics (see, e.g., \cite{Blackwell_Ferguson,
Polyaurn_coalescent,
Urn_population_genetics,
Pemantle_A_survey,
Mahamud_Polya_urn,
two_color_randomly_reinforced_urn_design,
BaTh14, 
bandyopadhyay2017polya,
Dynamics_of_an_adaptive_randomly_reinforced_urn,
Nonparametric-covariate-adjusted-response-adaptive-design,
Bandyooadhyay_Kaur2018,
Urn_models_response_adaptive_randomized_designs,
Terenin_Polya_urn,
bandyopadhyay2020strong,
bandyopadhyay2022new,
Urn_preferential_attachment,
Urnmodel_social_dynamics}).

 In recent years, there has been growing interest in \emph{interacting} urn models (see, e.g., \cite{launay2011interacting,launay2012generalized,Interactingfriedmanurn-1,kaur2023interacting,Completefirstorder-7,yogesh2024interacting, kaur2025interacting}), where multiple urns evolve simultaneously and interact through an underlying network. Such models provide a natural framework for studying reinforcement in interacting systems and various collective phenomena such as synchronization, consensus formation, and polarization.

A typical interacting urn model consists of a collection of urns indexed by the vertices of a graph. At each discrete time step, one ball is drawn from every urn according to a probability law determined by the current state of the system. Each urn is then reinforced by adding balls whose numbers and colors depend on the observed draws according to a prescribed interaction rule. Interacting urn models can broadly be divided into two classes. In the first class, the drawing probabilities depend on the global network configuration, whereas the reinforcement at each urn depends only on the color drawn from that urn \cite{launay2011interacting,launay2012generalized,Interactingfriedmanurn-1,Synchronisation-2,Functionalcltfornetwork-3,Empiricalmean-4,Weightedempiricalmean-5,Completefirstorder-7,Polarization-8}. In the second class, introduced in \cite{kaur2023interacting,yogesh2024interacting,kaur2025interacting}, the drawing probability depends only on the composition of the individual urn, while the reinforcement at each urn depends on the colors drawn from its neighboring urns. Although the interaction mechanisms are different, both classes share the important property that the conditional expectation of the color indicators can be written as a linear function of the current urn proportions.

In all these models, the underlying interaction graph is assumed to be finite. This allows the use of martingale methods, Perron--Frobenius theory, and stochastic approximation techniques to study the first- and second-order asymptotic behavior of the urn proportions. Much less is known when the interaction graph is infinite, since many of these techniques do not apply directly.

In this article, we study interacting urn schemes on a class of directed graphs that may have infinitely many vertices. Specifically, we consider \emph{finite ancestral directed acyclic graphs} (FA-DAGs), that is, directed acyclic graphs in which every vertex has only finitely many ancestors. This class includes many important infinite graphs, such as infinite rooted trees, infinite rectangular lattices with a fixed orientation, and several other directed structures. Throughout the article, we use the interaction-through-reinforcement model introduced in \cite{kaur2023interacting,yogesh2024interacting,kaur2025interacting}, where each urn is sampled according to its own composition, while reinforcement depends on the colors drawn from neighboring urns.

Our interest in infinite interacting urn models is partly motivated by the search for analytically tractable stochastic systems that exhibit features similar to \emph{self-organized criticality (SOC)} (see \cite{bak1987self,pruessner2012self} for details). Although SOC is a central topic in statistical physics, obtaining rigorous mathematical results for classical models such as the Abelian sandpile \cite{bak1987self} remains difficult, especially in higher dimensions (see \cite{frigg2003self,dhar1993self,pruessner2012self} for further discussion). For our models, the FA-DAG structure provides enough locality to make a rigorous analysis possible while still allowing the graph to be infinite.

Using the finite ancestral property of the graph, we prove almost sure convergence of the urn proportions and obtain explicit descriptions of the limiting configurations for both finite and infinite graphs. We also establish second-order asymptotic results for the color-proportions of individual urns and those of their ancestors, extending several known results for finite interacting urn models to this class of infinite graphs.

\subsection*{Notations used in the article}

For any set $S$, let $|S|$ denote its cardinality. For a positive integer $k$, we denote the $k$-dimensional Euclidean space by $\mathbb{R}^k$ and regard each element as a row vector with real entries. For $x\in\mathbb{R}^k$, let $x(j)$ denote its $j$-th coordinate. We denote the canonical basis vectors of $\mathbb{R}^k$ by $e_1,e_2,\ldots,e_k$, where $e_j$ is the row vector whose $j$-th entry is $1$ and all other entries are $0$. The $k$-dimensional row vector with all entries equal to $1$ is denoted by $\mathbbm{1}_k$; when the dimension is clear from the context, we will simply write $\mathbbm{1}$.

For a matrix $A$, let $A'$ denote its transpose and let $Sp(A)$ denote the set of its eigenvalues. We denote the $k\times k$ identity matrix by $I_k$. For matrices $A_1,\ldots,A_n$, let $\operatorname{diag}(A_1,\ldots,A_n)$ denote the block-diagonal matrix whose diagonal blocks are $A_1,\ldots,A_n$. The space of all real $m\times n$ matrices is denoted by $\mathbb{R}^{m\times n}$.

Unless stated otherwise, all convergences are with respect to the standard Euclidean topology on $\mathbb{R}^k$. Occasionally, we also consider the standard topology on $\mathbb{C}^k$. Convergence in $\mathbb{R}^{m\times n}$ is understood with respect to the operator norm. We denote the $\ell^2$-norm on $\mathbb{C}^k$ by $\|\cdot\|$. The notations $\xrightarrow{a.s.}$ and $\xrightarrow{d}$ stand for almost sure convergence and convergence in distribution respectively.

For $\mu\in\mathbb{R}^k$ and a positive semidefinite matrix $\Sigma\in\mathbb{R}^{k\times k}$, we write $N_k(\mu,\Sigma)$ for the $k$-dimensional normal (or singular normal) distribution with mean vector $\mu$ and positive define 
(or respectively non-negative definite)
covariance matrix
$\Sigma$. 

For two deterministic real sequences $\{a_n\}_{n\ge1}$ and $\{b_n\}_{n\ge1}$, we write $a_n=O(b_n)$ if there exist constants $c>0$ and $n_0\ge1$ such that
\[
|a_n|\le c|b_n|,\qquad \text{for all } n\ge n_0.
\]
Similarly, for a sequence of random vectors $\{X_n\}_{n\ge1}$, we write
\[
X_n=O(a_n)\quad {a.s.}
\]
if there exists a finite non-negative random variable $C$ such that
\[
\Q\bigl(\|X_n\|>C|a_n|\ \text{i.o.}\bigr)=0,
\]
where ``i.o.'' stands for ``infinitely often.''

\subsection*{Organization of the article} The article is constructed as follows. \autoref{sec:model} introduces the general model and network structure, \autoref{sec:results} states the main results, and \autoref{sec:proofs} contains the proofs.

\section{Model Description}\label{sec:model}

Let $\G=(V,\E)$ be a directed graph with nonempty vertex set $V$ and edge set
$\E\subseteq V\times V$. We say that $\G$ is \emph{finite} or \emph{infinite}
according as $V$ is finite or infinite. For each $v\in V$, define
\[
d_v^{in}:=\big|\{u\in V:(u,v)\in\E\}\big|,
\qquad
d_v^{out}:=\big|\{u\in V:(v,u)\in\E\}\big|,
\]
which are referred to as the \emph{in-degree} and \emph{out-degree} of $v$,
respectively. We further define
\[
d_v:=d_v^{in}+d_v^{out}
\]
and refer to $d_v$ as the \emph{degree} of $v$. If $\G$ is infinite, we
assume throughout that it is \emph{locally finite}, that is,
$d_v<\infty$ for every $v\in V$.

For $u,v\in V$, we write $u\rightsquigarrow v$ whenever $(u,v)\in\E$.
Given vertices $v_1,\ldots,v_n\in V$, $n\ge2$, such that
$v_j\rightsquigarrow v_{j+1}$ for $j=1,\ldots,n-1$, we write $v_1v_2\ldots v_n$ to denote the directed path
\[
v_1\rightsquigarrow v_2\rightsquigarrow\cdots
\rightsquigarrow v_n.
\]
We write
$\rt u v$ if there exists a directed path from $u$ to $v$. 

For our model, we place an urn at each vertex of $\G$. Each urn contains
balls of $k\ge2$ colors, indexed by $1,\ldots,k$. We denote the initial
configuration of the urn at vertex $v$ by
\[
\C{0}{v}
=
\bigl(\C{0}{v}(1),\ldots,\C{0}{v}(k)\bigr),
\]
where $\C{0}{v}(j)$ denotes the initial number of balls of color $j$.
For every directed edge $\e\in\E$, let $R^{\e}$ be a real $k\times k$
matrix, referred to as the \emph{reinforcement matrix} associated with
the edge $\e$.

The initial configurations $\{\C{0}{v}:v\in V\}$ and the reinforcement
matrices $\{R^{\e}:\e\in\E\}$ may be random. We assume, however, that they
have nonnegative entries almost surely. Furthermore, we assume that
\[
\C{0}{v}\bb{1}'>0
\qquad\text{a.s. for every }v\in V,
\]
and that every row of $R^{\e}$ is nonzero almost surely. These assumptions
ensure that drawing can be done from every urn at every time and that every color
can receive reinforcement.

We denote the configuration of the urn at vertex $v$ after $n$ updates by
\[
\C{n}{v}
=
\bigl(\C{n}{v}(1),\ldots,\C{n}{v}(k)\bigr).
\]

We consider two evolution schemes, referred to as the
\emph{sum-reinforcement scheme} and the \emph{average-reinforcement
scheme}.

\subsection*{Sum-reinforcement scheme}

The urn configurations evolve according to
\begin{equation}\label{2}
\C{n+1}{v}
=
\begin{cases}
\displaystyle
\C{n}{v}
+
\sum_{u:u\rightsquigarrow v}
\Z{n+1}{u}R^{(u,v)},
&\text{if}~d_v^{in}>0,\\[2mm]
\C{n}{v},
&\text{if}~d_v^{in}=0.
\end{cases}
\end{equation}
Here, $\Z{n+1}{v}$ denotes the color of the ball drawn from the urn at
vertex $v$ at time $n+1$, viewed as a random vector taking values in
$\{e_1,\ldots,e_k\}$, where $e_1,\ldots,e_k$ are the canonical basis
vectors of $\mathbb{R}^k$.

Let
\[
\mathcal{F}_n
:=
\sigma\left(
\C{j}{u},R^{\e}:
j\le n,\ u\in V,\ \e\in\E
\right).
\]
Conditionally on $\mathcal{F}_n$, the family
$\{\Z{n+1}{v}:v\in V\}$ is independent. More precisely, for every
nonempty finite set $V_0\subseteq V$ and every collection
$\{\z{}{v}\in\{e_1,\ldots,e_k\}:v\in V_0\}$,
\[
\Q\left(
\bigcap_{v\in V_0}
\{\Z{n+1}{v}=\z{}{v}\}
\,\middle|\,\mathcal{F}_n
\right)
=
\prod_{v\in V_0}
\frac{
\C{n}{v}(\z{}{v})'
}{
\C{n}{v}\bb{1}'
}.
\]

\subsection*{Average-reinforcement scheme}

In the average-reinforcement scheme, the urn configurations evolve
according to
\begin{equation}\label{3}
\C{n+1}{v}
=
\begin{cases}
\displaystyle
\C{n}{v}
+
\frac{1}{d_v^{in}}
\sum_{u\rightsquigarrow v}
\Z{n+1}{u}R^{(u,v)},
&\text{if}~d_v^{in}>0,\\[2mm]
\C{n}{v},
&\text{if}~d_v^{in}=0.
\end{cases}
\end{equation}
The drawing mechanism is the same as in the sum-reinforcement scheme:
conditionally on $\mathcal{F}_n$, the family
$\{\Z{n+1}{v}:v\in V\}$ is independent and, for each $v\in V$,
\[
\Q\left(
\Z{n+1}{v}=e_j\mid\mathcal{F}_n
\right)
=
\frac{\C{n}{v}(j)}
{\C{n}{v}\bb{1}'},
\qquad j=1,\ldots,k.
\]
Thus, the two schemes differ only in the amount of reinforcement added to
each urn: in the sum-reinforcement scheme, the contributions from all
in-neighbors are added directly, whereas in the average-reinforcement
scheme, their average is added.

\subsection{Interpretation of the models}

At each time step, one ball is drawn independently from every urn according to
its current configuration. The urns are then reinforced simultaneously. For a
vertex $v$, every in-neighbor $u$ satisfying $\rs{u}{v}$ contributes to the
reinforcement of the urn at $v$. If $\Z{n+1}{u}=e_j$, then in the sum-reinforcement scheme,
$R^{(u,v)}_{j,l}$ balls of color $l$ are added to the urn at $v$, whereas in
the average-reinforcement scheme,
$R^{(u,v)}_{j,l}/d_v^{\mathrm{in}}$ balls of color $l$ are added, for
$l=1,\ldots,k$. These updates are performed simultaneously and independently at all the vertices.
Fractional reinforcements are allowed and are consistent with the evolution
equations \eqref{2} and \eqref{3}, as is standard in the classical urn literature.

\subsection{Comparison between the two models}

The two models differ only in the way reinforcements from the in-neighbors are
aggregated. One uses their sum, whereas the other uses their average.
Consequently, for a fixed graph and a fixed collection of reinforcement
matrices, the two models may exhibit different asymptotic behavior. However,  
note that 
the second model can always be obtained from
the first model by 
simply replacing each replacement matrx
$R^{\e}$ by 
$R^{\e}/d_v^{\mathrm{in}}$. In particular, when $d_v^{in}=1$ for every
$v\in V$, the two models coincide.
In this article, we thus present our results only for the sum-reinforcement scheme.
The corresponding results for the average-reinforcement scheme follow immediately by replacing each reinforcement matrix $R^{\e}$ with $R^{\e}/d_v^{\mathrm{in}}$.

\subsection{Finite Ancestral Directed Acyclic Graphs (FA-DAGs)}
Throughout this article, we assume that the underlying graph $\G$ belongs to the following class, which we call \emph{finite ancestral directed acyclic graphs (FA-DAG).}

\begin{enumerate}[label=\textbf{(A\arabic*)}, ref=\textbf{(A\arabic*)}]
\item\label{a1}
$\G$ is a \emph{directed acyclic graph (DAG)} and 
we allow possibility of directed 
self-loops.\\

\item\label{a2}
Every vertex has only finitely many ancestors; that is, for each $v\in V$, the ancestor set,  
\[
A_v:=\{u\in V:u\rightarrow v\}
\]
is a finite set.\\

\item\label{a3}
Every vertex with a self-loop has no in-neighbor other than itself; equivalently,
\[
(v,v)\in\E
\quad\Longrightarrow\quad
d_v^{in}=1.
\]
\end{enumerate}

By \emph{acyclic} we mean that there exists no directed path
$v_1v_2\ldots v_n$ with $n\ge3$ consisting of distinct vertices
$v_1,v_2,\ldots,v_{n-1}\in V$ and satisfying $v_n=v_1$. Equivalently,
the graph $\G$ contains no directed cycles except possibly self-loops. 
\begin{definition}
A vertex $v\in V$ is called a generator vertex if $A_v=\{v\}$, and
a stubborn vertex if $A_v=\emptyset$. 
\end{definition}
Let $G$ and $S$ denote the sets of generator and stubborn vertices of $\G$, respectively.
Note that assumtions 
$\ref{a2}$ and $\ref{a3}$ imply that the
oldest ancestors of any vertex must be either a
\emph{generator vertex} or a 
\emph{stubborn vertex}. 
Thus $G\cup S\neq\emptyset$. However, both 
$G$ and $S$ can be infinite sets. \\ 

Following are some simple examples of FA-DAGs:

\begin{figure}[H]
    \centering
    \begin{tikzpicture}
        \Vertex[label=$0$,size=0.5]{a}
        \Vertex[x=2,label=$1$,size=0.5]{b}
        \Vertex[x=4,label=$2$,size=0.5]{c}
        \Vertex[x=6,label=$3$,size=0.5]{d}
        \Vertex[x=8,label=$4$,size=0.5]{e}
        \Vertex[x=10,label=$5$,size=0.5]{f}

        \Edge[loopposition=100,Direct](a)(a)
        \Edge[Direct](a)(b)
        \Edge[Direct](b)(c)
        \Edge[Direct](c)(d)
        \Edge[Direct](d)(e)
        \Edge[Direct](e)(f)

        \node at (11,0) {$\cdots\cdots$};
    \end{tikzpicture}
    \caption{}
    \label{fig:infinite_line}
\end{figure}

\begin{figure}[H]
    \centering
    \begin{minipage}{0.48\textwidth}
        \centering
        \begin{tikzpicture}
            \Vertex[x=3,label=$0$,size=0.5]{a}
            \Vertex[x=2,y=-1,label=$1$,size=0.5]{b}
            \Vertex[x=4,y=-1,label=$2$,size=0.5]{c}
            \Vertex[x=1.2,y=-2,label=$3$,size=0.5]{d}
            \Vertex[x=2.6,y=-2,label=$4$,size=0.5]{e}
            \Vertex[x=3.4,y=-2,label=$5$,size=0.5]{f}
            \Vertex[x=4.8,y=-2,label=$6$,size=0.5]{g}
            \Vertex[x=0.8,y=-3,label=$7$,size=0.5]{h}
            \Vertex[x=1.6,y=-3,label=$8$,size=0.5]{i}
            \Vertex[x=2.6,y=-3,label=$9$,size=0.5]{j}
            \Vertex[x=4.4,y=-3,label=$10$,size=0.5]{k}
            \Vertex[x=5.2,y=-3,label=$11$,size=0.5]{l}

            \Edge[loopposition=100,Direct](a)(a)
            \Edge[Direct](a)(b)
            \Edge[Direct](a)(c)
            \Edge[Direct](b)(d)
            \Edge[Direct](b)(e)
            \Edge[Direct](c)(f)
            \Edge[Direct](c)(g)
            \Edge[Direct](d)(h)
            \Edge[Direct](d)(i)
            \Edge[Direct](e)(j)
            \Edge[Direct](g)(k)
            \Edge[Direct](g)(l)

            \node at (0.8,-3.7) {$\vdots$};
            \node at (1.6,-3.7) {$\vdots$};
            \node at (2.6,-3.7) {$\vdots$};
            \node at (4.4,-3.7) {$\vdots$};
            \node at (5.2,-3.7) {$\vdots$};
        \end{tikzpicture}
        \caption{}
        \label{fig:figure2}
    \end{minipage}
    \hfill
    \begin{minipage}{0.48\textwidth}
        \centering
        \begin{tikzpicture}
            \Vertex[x=3,label=$0$,size=0.5]{m}
            \Vertex[x=2,y=-1,label=$1$,size=0.5]{n}
            \Vertex[x=4,y=-1,label=$2$,size=0.5]{o}
            \Vertex[x=1.2,y=-2,label=$3$,size=0.5]{p}
            \Vertex[x=3,y=-2,label=$4$,size=0.5]{q}
            \Vertex[x=4.8,y=-2,label=$5$,size=0.5]{r}
            \Vertex[x=0.8,y=-3,label=$6$,size=0.5]{s}
            \Vertex[x=2.2,y=-3,label=$7$,size=0.5]{t}
            \Vertex[x=3.8,y=-3,label=$8$,size=0.5]{u}
            \Vertex[x=5.2,y=-3,label=$9$,size=0.5]{v}

            \Edge[loopposition=100,Direct](m)(m)
            \Edge[Direct](m)(n)
            \Edge[Direct](m)(o)
            \Edge[Direct](n)(p)
            \Edge[Direct](n)(q)
            \Edge[Direct](o)(q)
            \Edge[Direct](o)(r)
            \Edge[Direct](p)(s)
            \Edge[Direct](p)(t)
            \Edge[Direct](q)(t)
            \Edge[Direct](q)(u)
            \Edge[Direct](r)(u)
            \Edge[Direct](r)(v)

            \node at (0.8,-3.7) {$\vdots$};
            \node at (2.2,-3.7) {$\vdots$};
            \node at (3.8,-3.7) {$\vdots$};
            \node at (5.2,-3.7) {$\vdots$};
        \end{tikzpicture}
        \caption{}
        \label{fig:figure3}
    \end{minipage}
\end{figure}

\begin{figure}[H]
    \centering
    \begin{tikzpicture}

        \Vertex[label=$0$,size=0.5]{a}
        \Vertex[x=2,label=$1$,size=0.5]{b}
        \Vertex[x=4,label=$2$,size=0.5]{c}
        \Vertex[x=6,label=$3$,size=0.5]{d}
        \Vertex[x=8,label=$4$,size=0.5]{e}
        \Vertex[x=10,label=$5$,size=0.5]{f}

        \Vertex[y=-1,label=$8$,size=0.5]{i}
        \Vertex[x=2,y=-1,label=$9$,size=0.5]{j}
        \Vertex[x=4,y=-1,label=$10$,size=0.5]{k}
        \Vertex[x=6,y=-1,label=$11$,size=0.5]{l}
        \Vertex[x=8,y=-1,label=$12$,size=0.5]{m}
        \Vertex[x=10,y=-1,label=$13$,size=0.5]{n}

        \Vertex[y=-2,label=$16$,size=0.5]{q}
        \Vertex[x=2,y=-2,label=$17$,size=0.5]{r}
        \Vertex[x=4,y=-2,label=$18$,size=0.5]{s}
        \Vertex[x=6,y=-2,label=$19$,size=0.5]{t}
        \Vertex[x=8,y=-2,label=$20$,size=0.5]{u}
        \Vertex[x=10,y=-2,label=$21$,size=0.5]{v}

        \Vertex[y=-3,label=$24$,size=0.5]{q1}
        \Vertex[x=2,y=-3,label=$25$,size=0.5]{r1}
        \Vertex[x=4,y=-3,label=$26$,size=0.5]{s1}
        \Vertex[x=6,y=-3,label=$27$,size=0.5]{t1}
        \Vertex[x=8,y=-3,label=$28$,size=0.5]{u1}
        \Vertex[x=10,y=-3,label=$29$,size=0.5]{v1}

        \Edge[loopposition=100,Direct](a)(a)
        \Edge[Direct](a)(b)
        \Edge[Direct](b)(c)
        \Edge[Direct](c)(d)
        \Edge[Direct](d)(e)
        \Edge[Direct](e)(f)

        \Edge[Direct](i)(j)
        \Edge[Direct](j)(k)
        \Edge[Direct](k)(l)
        \Edge[Direct](l)(m)
        \Edge[Direct](m)(n)

        \Edge[Direct](q)(r)
        \Edge[Direct](r)(s)
        \Edge[Direct](s)(t)
        \Edge[Direct](t)(u)
        \Edge[Direct](u)(v)

        \Edge[Direct](q1)(r1)
        \Edge[Direct](r1)(s1)
        \Edge[Direct](s1)(t1)
        \Edge[Direct](t1)(u1)
        \Edge[Direct](u1)(v1)

        \Edge[Direct](a)(i)
        \Edge[Direct](b)(j)
        \Edge[Direct](c)(k)
        \Edge[Direct](d)(l)
        \Edge[Direct](e)(m)
        \Edge[Direct](f)(n)

        \Edge[Direct](i)(q)
        \Edge[Direct](j)(r)
        \Edge[Direct](k)(s)
        \Edge[Direct](l)(t)
        \Edge[Direct](m)(u)
        \Edge[Direct](n)(v)

        \Edge[Direct](q)(q1)
        \Edge[Direct](r)(r1)
        \Edge[Direct](s)(s1)
        \Edge[Direct](t)(t1)
        \Edge[Direct](u)(u1)
        \Edge[Direct](v)(v1)

        \node at (10.8,0) {$\cdots$};
        \node at (10.8,-1) {$\cdots$};
        \node at (10.8,-2) {$\cdots$};
        \node at (10.8,-3) {$\cdots$};

        \node at (0,-3.5) {$\vdots$};
        \node at (2,-3.5) {$\vdots$};
        \node at (4,-3.5) {$\vdots$};
        \node at (6,-3.5) {$\vdots$};
        \node at (8,-3.5) {$\vdots$};
        \node at (10,-3.5) {$\vdots$};

        \node at (10.5,-3.5) {$\ddots$};

    \end{tikzpicture}
    \caption{}
    \label{fig:graph4}
\end{figure}

\begin{figure}[H]
\centering
\begin{tikzpicture}

\Vertex[x=5,label=$0$,size=0.5]{a}
\Vertex[x=4,y=-1,label=$1$,size=0.5]{b}
\Vertex[x=6,y=-1,label=$2$,size=0.5]{c}
\Vertex[x=3,y=-2,label=$3$,size=0.5]{d}
\Vertex[x=2.5,y=-3,label=$7$,size=0.5]{h}
\Vertex[x=3.5,y=-3,label=$8$,size=0.5]{i}
\Vertex[x=4.5,y=-2,label=$4$,size=0.5]{e}
\Vertex[x=4.5,y=-3,label=$9$,size=0.5]{j}
\Vertex[x=5.5,y=-2,label=$5$,size=0.5]{f}
\Vertex[x=7,y=-2,label=$6$,size=0.5]{g}
\Vertex[x=8,y=-3,label=$11$,size=0.5]{k}
\Vertex[x=6,y=-3,label=$10$,size=0.5]{l}

\Vertex[x=6,label=$12$,size=0.5]{a1}
\Vertex[x=7,y=-1,label=$13$,size=0.5]{b1}
\Vertex[x=8,y=-1,label=$14$,size=0.5]{c1}
\Vertex[x=8,y=-2,label=$15$,size=0.5]{d1}
\Vertex[x=8,label=$16$,size=0.5]{e1}
\Vertex[x=9,label=$17$,size=0.5]{f1}
\Vertex[x=9,y=-1,label=$18$,size=0.5]{g1}
\Vertex[x=9,y=-2,label=$19$,size=0.5]{h1}

\Edge[loopposition=100,Direct](a)(a)
\Edge[Direct](a)(b)
\Edge[Direct](a)(c)
\Edge[Direct](b)(d)
\Edge[Direct](b)(e)
\Edge[Direct](c)(f)
\Edge[Direct](c)(g)
\Edge[Direct](d)(h)
\Edge[Direct](d)(i)
\Edge[Direct](e)(j)
\Edge[Direct](g)(k)
\Edge[Direct](g)(l)
\Edge[Direct](e)(f)
\Edge[Direct](f)(l)
\Edge[Direct](a)(f)
\Edge[Direct](b)(i)
\Edge[Direct](f)(j)
\Edge[Direct](i)(j)

\Edge[Direct](a1)(c)
\Edge[Direct](a1)(e)
\Edge[Direct](b1)(f)
\Edge[Direct](b1)(g)
\Edge[loopposition=100,Direct](b1)(b1)
\Edge[Direct](b1)(c1)
\Edge[Direct](c1)(d1)
\Edge[Direct](c1)(g)
\Edge[Direct](e1)(c1)
\Edge[loopposition=100,Direct](f1)(f1)
\Edge[Direct](f1)(c1)
\Edge[Direct](f1)(g1)
\Edge[Direct](h1)(d1)
\Edge[Direct](h1)(k)

\node at (9.7,0) {$\cdots$};
\node at (9.7,-1) {$\cdots$};
\node at (9.7,-2) {$\cdots$};
\node at (9.7,-3) {$\cdots$};

\node at (2.5,-3.5) {$\vdots$};
\node at (3.5,-3.5) {$\vdots$};
\node at (4.5,-3.5) {$\vdots$};
\node at (5.5,-3.5) {$\vdots$};
\node at (6.5,-3.5) {$\vdots$};
\node at (7.5,-3.5) {$\vdots$};
\node at (8.5,-3.5) {$\vdots$};
\node at (9.5,-3.5) {$\vdots$};
\end{tikzpicture}
\caption{}
\label{fig:graph5}
\end{figure}

{\footnotesize \autoref{fig:infinite_line} is the \emph{one-sided infinite line with a single generator}(0). \autoref{fig:figure2} is an \emph{infinite rooted tree with directions from root toward leaves} with root(0) as the generator. \autoref{fig:figure3} is an \emph{infinite Pascal-like triangle} with a single generator(0). \autoref{fig:graph4} is an \emph{infinite rectangular grid} with a single generator(0). \autoref{fig:graph5} contains multiple generators (0,13,17) and stubborn vertices (12,16,19).}

In this article, we are primarily interested in infinite FA-DAGs, although all of
our results apply equally to finite graphs. For $v\in V$, let
\[
\U{n}{v}:=\frac{\C{n}{v}}{\C{n}{v}\bb{1}'}
\]
denote the vector of color-proportions in the urn at vertex $v$ at time $n$, and
let
\[
\N{n}{v}:=\frac{1}{n}\sum_{i=1}^n \Z{i}{v}
\]
denote the empirical \emph{color-count statistics}, representing the proportions of
colors drawn from the urn at vertex $v$ up to time $n$.

Our objective is to study the joint asymptotic behavior of the family of urn-proportions
\[
\{\U{n}{v}: v\in V,\ n\ge0\}.
\]

Note that if $\G$ is infinite then at each time infinitely many urns are updating simultaneously, so there may not be an Athreya-Karlin type continuous embedding of the process (see \cite{athreya1968embedding}) which is 
in general, a very important tool to analyze such  reinforced stochastic processes.

\section{Main results}\label{sec:results}
\subsection{Strong convergence of urn proportions}
It is worthwhile to note that the generator urns are nothing but the usual generalized P\'{o}lya type urns with replacement matrices corresponding to their self-loops. These urns are studied in the literature for more than a decade and it is well known that under some tenability conditions the urn proportion always converges (see, e.g.,\cite{gouet1989martingale,gouet1993martingale,dasgupta2011strong}). If the FA-DAG has at least one generator vertex then we make the following assumption.

\begin{assumption}\label{asm:strongconvergence}
Color-count statistics corresponding to all generator urns converge almost surely. That is for every generator $g\in G$ there exists a random vector $\N{\infty}{g}$ such that
$$\N{n}{g}\xrightarrow{a.s.}\N{\infty}{g}$$ 
\end{assumption}

\begin{remark}
Assumption \ref{asm:strongconvergence} covers a large class of urns, in particular it holds for any balanced urn (see \cite{gouet1989martingale},\cite{gouet1997strong},\cite{dasgupta2011strong}). The restriction of balancedness just helps to sort out the martingales corresponding to the irreducible blocks of the reinforcement matrix which makes it easy to establish convergence. However the martingale approach does not work when the reinforcement matrix is not balanced. The famous Athreya-Karlin embedding of urns leads to a continuous time Markov branching process and helps to establish the convergence for irreducible reinforcement matrices without the restriction of balancedness (see \cite{athreya1968embedding}). Recently Janson proved the convergence for any triangular urn (see \cite{janson2024almost}) using the embedding technique and the theory of general continuous time Markov branching process with continuous state-space, which is, to our knowledge, the best known result for unbalanced replacements. Assumption \ref{asm:strongconvergence} is believed to be true for any finite urn, since the colors can always be arranged in such a way that the corresponding reinforcement matrix becomes upper block triangular with irreducible diagonal blocks and then it seems the techniques used in \cite{janson2024almost} along with the techniques of irreducible urns can be used, although it is still open for further research.
\end{remark}

\begin{remark}
Note that Assumption \ref{asm:strongconvergence} is equivalent to saying that $\U{n}{g}$ converges almost surely for all $g\in G$, since 
\begin{equation}
\U{n}{g}=\frac{\N{n}{g}R^{(g,g)}}{\N{n}{g}R^{(g,g)}\mathbbm{1}'}
\end{equation} 
and by a martingale version of Borel-Cantelli lemma it can be shown that (see \cite{gouet1997strong}) 
\begin{equation}
\N{n}{g}=(I_k+o(n))\frac{1}{n}\sum\limits_{j=0}^n\U{j}{g}
\end{equation}
where $o(n)$ is a diagonal matrix whose entries converge to $0$ almost surely.
\end{remark}

The following result on almost sure asymptotic behavior of urn proportions holds for very general set-up.

\begin{theorem}[\textbf{Strong convergence of urn proportions}]
\label{thm:strongconvergence}
Suppose that the underlying graph $\G$ is an FA-DAG and that
Assumption~\ref{asm:strongconvergence} holds. Then all urn-proportions
$\{\U{n}{v}:n\geq0,\ v\in V\}$ converge almost surely. More precisely, for
each $v\in V$, there exists a $k$-dimensional random vector
$\U{\infty}{v}$ such that
\begin{equation}
    \U{n}{v}\xrightarrow{a.s.}\U{\infty}{v},
    \qquad\text{as }n\rightarrow\infty.
\end{equation}

For the sum-reinforcement scheme, the limiting proportions are
characterized recursively by
\begin{equation}\label{lim}
\U{\infty}{v}
=
\begin{cases}
\U{0}{v},
& \text{if }v\in S,\\[0.4cm]
\displaystyle
\frac{
\sum\limits_{\substack{u:u\rightsquigarrow v,u\in G}}
\N{\infty}{u}R^{(u,v)}
+
\sum\limits_{\substack{u:u\rightsquigarrow v,u\in V\setminus G}}
\U{\infty}{u}R^{(u,v)}
}{
\sum\limits_{\substack{u:u\rightsquigarrow v,u\in G}}
\N{\infty}{u}R^{(u,v)}\bb{1}'
+
\sum\limits_{\substack{u:u\rightsquigarrow v,u\in V\setminus G}}
\U{\infty}{u}R^{(u,v)}\bb{1}'
},
& \text{if }v\in V\setminus S.
\end{cases}
\end{equation}

For the average-reinforcement scheme, the limiting proportions are
characterized by the same recursion, with each reinforcement matrix
$R^{(u,v)}$ replaced by $R^{(u,v)}/d_v^{\mathrm{in}}$.
\end{theorem}

\begin{remark}
\autoref{thm:strongconvergence} holds for any initial configuration $\{\C{0}{v} : v\in V\}$ and reinforcement matrices $\{R^{\e} : \e\in\E\}$ subject to the condition 
that initially all urns are non-empty (that is 
$\C{0}{v}\bb{1}'>0$ almost surely for all $v\in V$) and replacements are non-trivial (that is all $R^{\e}$'s have almost surely non-zero rows).    
\end{remark}

\begin{remark}
Assumption~\ref{asm:strongconvergence} is needed only when generator vertices are present. In the absence of generators, i.e., when \(G=\emptyset\), \autoref{thm:strongconvergence} guarantees almost sure convergence of all urns regardless.
\end{remark}
    
\begin{corollary}[\textbf{Synchronization to a random limit}]\label{cor:synchronization} 
Consider the sum-reinforcement scheme and suppose that $\G$ has a unique generator
vertex $g$ and no stubborn vertices, that is, $G\cup S=G=\{g\}$. Assume that all
reinforcement matrices are of P\'olya type, that is, $R^{\e}=I_k$ for all $\e\in\E$,
and that the initial configurations are non-random with
$\N{0}{g}=(n_1,n_2,\ldots,n_k)$, where $n_i>0$ for all $i$. Then, for every
$v\in V$,
\[
\U{n}{v}\xrightarrow[]{a.s.}\U{\infty}{g},
\]
where $\U{\infty}{g}$ has a Dirichlet distribution with parameter
$(n_1,n_2,\ldots,n_k)$.
\end{corollary}

\begin{remark}
   Corollary~\ref{cor:synchronization} covers a broad class of models, which may be viewed as Pólya urn schemes on DAGs with a root generator. The examples in Figures~\ref{fig:infinite_line}--\ref{fig:graph5} belong to this class, and in all such cases the urn-proportions synchronize and converge to a common Dirichlet-distributed random vector.
\end{remark}

\begin{remark}
    The interaction structure may be random, so that the reinforcement matrices
$\{R^{\e}:\e\in\E\}$ are random, possibly depending on the initial configurations.
Nevertheless, \autoref{thm:strongconvergence} remains valid under the conditional law given the
reinforcements. As an illustration, consider the one-sided infinite line with a single generator,
where $V=\{0,1,2,\ldots\}$ and
$\E=\{(0,0)\}\cup\{(n,n+1):n\in V\}$, and assume that $R^{(0,0)}$ is balanced.
Then all urn proportions converge almost surely, and for both models the limits are given by
\[
\U{\infty}{v}
=
\frac{\N{\infty}{0}R^{(0,1)}\cdots R^{(v-1,v)}}
{\N{\infty}{0}R^{(0,1)}\cdots R^{(v-1,v)}\bb{1}'},
\qquad \forall v\ge1,
\]
where $\N{\infty}{0}$ denotes the almost sure limit of the generator’s
color-count statistic.
\end{remark}

It must be clear from \autoref{thm:strongconvergence} that the asymptotic behavior of the entire system of urns is governed by the generator and stubborn urns only, for example in Corollary \ref{cor:synchronization} all urns converge to a random limit since the generator does it so, similarly the generator urns can lead the entire system to a deterministic limit, to demonstrate this we make the following assumptions:

\begin{assumption}\label{asm:deterministiclimit} Assume that 
\begin{enumerate}[label=\textbf{(B\arabic*)}, ref=\textbf{(B\arabic*)}] 
\item\label{b1} all initial urn configurations $\C{0}{v}$ and all reinforcement
matrices $R^{\e}$ are deterministic.\\
\item\label{b2} Each reinforcement matrix $R^{\e}$ is balanced, that is,
$R^{\e}\bb{1}'=r(\e)\bb{1}'$ for some $r(\e)>0$.\\
\item\label{b3} For every generator $g\in G$, the matrix $R^{(g,g)}$ is
irreducible.
\end{enumerate}
\end{assumption}

The following corollary is a direct consequence of \autoref{thm:strongconvergence}.
\begin{corollary}[\textbf{Strong convergence for irreducible generators}]
\label{cor:deterministiclimit}
Suppose that Assumption~\ref{asm:deterministiclimit} holds. Then the entire
system $\{\U{n}{v}:n\ge0,\ v\in V\}$ converges almost surely to a
deterministic limit $\{\us{}{v}:v\in V\}$; that is,
\[
\U{n}{v}\xrightarrow{a.s.}\us{}{v},
\qquad
\text{as }n\to\infty,\ \forall v\in V.
\]

For the sum-reinforcement scheme, the limiting proportions are
characterized as follows. For each generator $g\in G$, the limit
$\us{}{g}$ is the normalized Perron--Frobenius eigenvector of
$R^{(g,g)}$. For $v\in V\setminus G$,
\begin{equation}\label{11}
\us{}{v}:=
\begin{cases}
\U{0}{v}, & \text{if }v\in S,\\[0.25cm]
\displaystyle
\frac{1}{b_v}
\sum_{p:p\rightsquigarrow v}
\us{}{p}R^{(p,v)},
& \text{if }v\in V\setminus(G\cup S),
\end{cases}
\end{equation}
where
\[
b_v:=
\begin{cases}
\displaystyle
\sum_{p:p\rightsquigarrow v}r\big((p,v)\big),
& \text{if }v\in V\setminus S,\\[0.25cm]
1, & \text{if }v\in S,
\end{cases}
\]
is referred to as the \emph{total balance} of the vertex $v$.

For the average-reinforcement scheme, the limiting proportions are
obtained from \eqref{11} by replacing each reinforcement matrix
$R^{(p,v)}$ with $R^{(p,v)}/d_v^{\mathrm{in}}$.
\end{corollary}

\subsection{Connection of the strong convergence limit with an appropriate Dirichlet problem} In this part of the section, we will show that the limit expression \eqref{lim} in \autoref{thm:strongconvergence} is closely connected with standard Dirichlet problem on $\G$.

Let $A=G\cup S$ and define the probability simplex
\[
S_k:=\left\{x\in\mathbb{R}^k:x(i)\geq 0,\ \forall i=1,\ldots,k,\quad
x\mathbbm{1}'=1\right\}.
\]
Fix a collection $\{x_a\}_{a\in A}\subseteq S_k$.

The Dirichlet problem on the FA-DAG $\G$ with boundary condition
$\{x_a\}_{a\in A}$ is to find a sequence $\{z_v\}_{v\in V}\subseteq S_k$
satisfying
\begin{equation}\label{eq:dirichlet_problem}
z_v=
\begin{cases}
x_v, & v\in A,\\[0.15cm]
\displaystyle\frac{1}{d_v^{\mathrm{in}}}\sum_{u:u\rightsquigarrow v}z_u,
& v\in V\setminus A.
\end{cases}
\end{equation}

To characterize the solution, fix $v\in V\setminus A$ and consider the
Markov chain $\{Y_n^{(v)}\}_{n\geq0}$ on $V$, started from
$Y_0^{(v)}=v$, with transition probabilities
\begin{equation}\label{eq:backward_MC}
\Q\left(Y_{n+1}^{(v)}=x\mid Y_n^{(v)}=u\right)
=
\begin{cases}
\displaystyle\frac{1}{d_u^{\mathrm in}},
& u\in V\setminus A,\ x\rightsquigarrow u,\\[0.2cm]
1, & x=u\in A,\\
0, & \text{otherwise}.
\end{cases}
\end{equation}
Thus, whenever the chain is at a vertex outside $A$, it moves uniformly
to one of its in-neighbors, while every vertex in $A$ is absorbing.
Define the hitting time
\[
T_v:=\inf\left\{n\geq0:Y_n^{(v)}\in A\right\}.
\]
Since $\G$ is an FA-DAG, every vertex has only finitely many ancestors,
and hence
\[
\Q(T_v<\infty)=1.
\]
Consequently, we may define
\begin{equation}\label{eq:dirichlet_solution}
y_v:=
\begin{cases}
\displaystyle
\mathbb{E}\left[x_{Y_{T_v}^{(v)}}\right],
& \text{if}~v\in V\setminus A,\\[0.25cm]
x_v, & \text{if}~v\in A.
\end{cases}
\end{equation}
It follows from the first-step decomposition of the Markov chain that
\[
y_v=
\frac{1}{d_v^{\mathrm in}}\sum_{u:u\rightsquigarrow v}y_u,
\qquad \forall v\in V\setminus A,
\]
while $y_v=x_v$ for all $v\in A$. Hence, $\{y_v\}_{v\in V}$ solves the
Dirichlet problem \eqref{eq:dirichlet_problem}. Moreover, the solution is
unique.

We now relate this Dirichlet problem to the strong convergence result.
Suppose that, in the sum-reinforcement scheme, all reinforcement matrices
are of P\'olya type, i.e.,
\[
R^{\e}=I_k,\qquad\forall\e\in\E.
\]
Then, by the expression \eqref{lim}, the sequence $\{X_v\}_{v\in V}$ defined by
\begin{equation}\label{eq:limit_dirichlet}
X_v=
\begin{cases}
\N{\infty}{v}, & v\in G,\\
\U{0}{v}, & v\in S,\\
\U{\infty}{v}, & v\in V\setminus(G\cup S)
\end{cases}
\end{equation}
satisfies
\[
X_v=
\begin{cases}
\N{\infty}{v}, & v\in G,\\[0.1cm]
\U{0}{v}, & v\in S,\\[0.1cm]
\displaystyle\frac{1}{d_v^{\mathrm in}}
\sum_{u:u\rightsquigarrow v}X_u,
& v\in V\setminus(G\cup S).
\end{cases}
\]
Therefore, $\{X_v\}_{v\in V}$ is precisely the unique solution of the
Dirichlet problem on $\G$ with boundary $A=G\cup S$ and boundary values
\[
\left\{\N{\infty}{g}:g\in G\right\}
\cup
\left\{\U{0}{s}:s\in S\right\}.
\]

Equivalently, for every $v\in V\setminus(G\cup S)$,
\[
\U{\infty}{v}
=
\mathbb{E}\left[
X_{Y_{T_v}^{(v)}}
\right],
\]
where
\[
X_a=
\begin{cases}
\N{\infty}{a}, & a\in G,\\
\U{0}{a}, & a\in S.
\end{cases}
\]

Thus, in the Pólya reinforcement case, the strong convergence limit is precisely the harmonic extension of the random boundary values prescribed on $G\cup S$.

\subsection{Rate of convergence of urn proportions:}\label{subsec:rate} 
In the urn-model literature, it is well known that generalized P\'olya urns with
irreducible replacement matrices typically exhibit asymptotically normal
fluctuations around their almost sure limits under suitable conditions (see
\cite{janson2006limit}, \cite{laruelle2013randomized}).

In this part of the section, we analyze the asymptotic behavior of urn fluctuations around
their almost sure limits under Assumption~\ref{asm:deterministiclimit}. We mention that
condition~\ref{b1} is not essential for the results presented here, since
one may instead work under the regular conditional probability given the
\(\sigma\)-algebra \(
\sigma\big(\{\C{0}{v},R^{\e}: v\in V,\ \e\in\E\}\big).
\) Nevertheless, we impose this assumption to avoid unnecessary technical
complications.

Corollary~\ref{cor:deterministiclimit} provides a complete characterization of the
deterministic limiting configuration of the interacting urn system under
Assumption~\ref{asm:deterministiclimit}. Our interest now lies in the asymptotic
behavior of the fluctuations
\[
\U{n}{v}-\us{}{v}.
\]
More precisely, we aim to determine the joint asymptotic distribution of any urn
$v$ together with all its ancestor urns. To this end, we first introduce some
notation and terminology.

For a vertex $v\in V$, let $\G_v=(V_v,\E_v)$ denote the induced subgraph of $\G$
consisting of the vertex $v$ and all its ancestors, where
\[
V_v:=A_v\cup\{v\}.
\]
Note that $\G_v$ is a finite FA-DAG. Define
\[
G_v:=V_v\cap G
\qquad\text{and}\qquad
S_v:=V_v\cap S
\]
to be the sets of generator and stubborn ancestors of $v$, respectively.

We relabel the vertices of $\G_v$ as follows:
\begin{itemize}
\item[1.] Vertices in $G_v\cup S_v$ are relabelled as
$1,\ldots,|G_v\cup S_v|$, with generator vertices listed first, followed by
stubborn vertices.\\
\item[2.] The remaining vertices in $V_v\setminus(G_v\cup S_v)$, if any, are
relabeled as $|G_v\cup S_v|+1,\ldots,|V_v|$, in such a way that if $j$ is an
ancestor of $l$, then $j<l$.
\end{itemize}

Such a relabeling is possible due to the acyclic nature of $\G_v$. Under this
labeling, $j\le l$ whenever $j$ is an ancestor of $l$, with equality if and only
if $l$ is a generator vertex.

Let
\[
m_1:=|G_v|,\qquad m_2:=|S_v|,\qquad m_3:=|V_v\setminus(G_v\cup S_v)|,
\]
and set $m:=|V_v|=m_1+m_2+m_3$. If $m_3=0$, the system consists only of independent
generator and stubborn urns and there is nothing further to study. Hence, we
assume throughout that $m_3\ge1$.

Define the $mk$-dimensional random vector $U_n$ by
\[
U_n :=
\begin{cases}
\bigl(
\U{n}{1},\ldots,\U{n}{m_1},
\N{n}{m_1+1},\ldots,\N{n}{m_1+m_2},
\\[-0.05cm]
\quad\qquad\qquad\qquad\qquad
\U{n}{m_1+m_2+1},\ldots,\U{n}{m}
\bigr),
&\text{if}~m_1\neq0,\ m_2\neq0,\\[0.35cm]
\bigl(
\N{n}{1},\ldots,\N{n}{m_2},
\U{n}{m_2+1},\ldots,\U{n}{m}
\bigr),
&\text{if}~ m_1=0,\ m_2\neq0,\\[0.35cm]
\bigl(
\U{n}{1},\ldots,\U{n}{m_1},
\U{n}{m_1+m_2+1},\ldots,\U{n}{m}
\bigr),
&\text{if}~ m_1\neq0,\ m_2=0 .
\end{cases}
\]

To avoid unnecessary complications, we restrict attention to the sum-reinforcement
scheme. Under Assumption~\ref{asm:deterministiclimit},
Corollary~\ref{cor:deterministiclimit} implies that
\[
U_n \xrightarrow{a.s.} u := (u^{(1)},\ldots,u^{(m)}),
\]
where the limiting vectors $u^{(j)}$ are characterized as follows:
\begin{align}
u^{(j)}R^{(j,j)} &= r((j,j))u^{(j)},
&& 1\le j\le m_1, \label{u1}\\
u^{(j)} &= \U{0}{j},
&& m_1+1\le j\le m_1+m_2, \label{u2}\\
u^{(j)} &= \frac{1}{b_j}\sum_{l:l\rightsquigarrow j} u^{(l)}R^{(l,j)},
&& m_1+m_2+1\le j\le m. \label{u3}
\end{align}

Note that the original vertex $v$ corresponds to the last coordinate of $U_n$,
since it is relabeled as $m$.

For $l,j\in\{1,\ldots,m\}$, define the $k\times k$ matrices
\[
H^{(l,j)} :=
\begin{cases}
\dfrac{1}{b_j}R^{(l,j)}, & \text{if}~l\rightsquigarrow j,\\[0.25cm]
I_k, &\text{if}~m_1+1\le l,j\le m_1+m_2,\\[0.25cm]
0, & \text{otherwise},
\end{cases}
\]
where $0$ denotes the null matrix.

Assume now that $G_v\neq\emptyset$ (equivalently, $m_1\neq0$). Define
\[
\gamma :=
\max\left\{\Re(\lambda):
\lambda\in\bigcup_{j=1}^{m_1} Sp(H^{(j,j)})\setminus\{1\}\right\}.
\]
By the Perron--Frobenius theorem, $\gamma<1$, since each $H^{(j,j)}$ is irreducible
and row-stochastic.

For $1\le j\le m_1$ and $\lambda\in Sp(H^{(j,j)})$, let $\nu_j(\lambda)$ denote the
algebraic multiplicity of $\lambda$. For
$\lambda\in\bigcup_{j=1}^{m_1} Sp(H^{(j,j)})\setminus\{1\}$, define
\[
\nu(\lambda):=\sum_{\substack{j\le m_1,\lambda\in Sp(H^{(j,j)})}}\nu_j(\lambda),
\]
and finally set
\[
\nu :=
\max\left\{\nu(\lambda):
\lambda\in\bigcup_{j=1}^{m_1} Sp(H^{(j,j)}),
\ \Re(\lambda)=\gamma\right\}.
\]

We will determine the asymptotic distribution of the fluctuations $U_n-u$ in
terms of the parameters $\gamma$ and $\nu$. Before proceeding further, we introduce a few quantities that will be used to state the fluctuation results.

\subsubsection{\textbf{\emph{Interacting reinforcement matrix}}}
We define the $mk\times mk$ block matrix $H$ by setting its $(l,j)$-th block equal
to $H^{(l,j)}$ for all $l,j\in\{1,\ldots,m\}$. Depending on the values of $m_1$ and
$m_2$, the matrix $H$ admits the following block-upper-triangular representations.

If $m_1\neq0$ and $m_2\neq0$,
\[
H=
\begin{bmatrix}
\text{diag}(H^{(1,1)},\ldots,H^{(m_1,m_1)}) & 0 & H_1\\
0 & I_{m_2k} & H_2\\
0 & 0 & 0
\end{bmatrix}.
\]

If $m_1=0$ and $m_2\neq0$,
\[
H=
\begin{bmatrix}
I_{m_2k} & H_2\\
0 & 0
\end{bmatrix}.
\]

If $m_1\neq0$ and $m_2=0$,
\[
H=
\begin{bmatrix}
\text{diag}(H^{(1,1)},\ldots,H^{(m_1,m_1)}) & H_1\\
0 & 0
\end{bmatrix}.
\]

Here,
\[
H_1 :=
\begin{bmatrix}
H^{(1,m_1+m_2+1)} & \cdots & H^{(1,m)}\\
\vdots & \ddots & \vdots\\
H^{(m_1,m_1+m_2+1)} & \cdots & H^{(m_1,m)}
\end{bmatrix}
\]
and

\[
H_2 :=
\begin{bmatrix}
H^{(m_1+1,m_1+m_2+1)} & \cdots & H^{(m_1+1,m)}\\
\vdots & \ddots & \vdots\\
H^{(m_1+m_2,m_1+m_2+1)} & \cdots & H^{(m_1+m_2,m)}
\end{bmatrix},
\]
with the understanding that $H_1$ (resp.\ $H_2$) is defined only when $m_1\neq0$
(resp.\ $m_2\neq0$). All zero blocks in $H$ denote null matrices of appropriate
order. We call $H$ the \emph{interacting reinforcement matrix} associated with
the process $\{U_n:n\ge0\}$.

\subsubsection{\textbf{\emph{Exponent matrix}}}
We next define the $mk\times mk$ \emph{exponent matrix} $D_u$, which depends on the
limit $u$ and the matrix $H$. If $m_1\neq0$ and $m_2\neq0$, set
\[
D_u=
\begin{bmatrix}
D & 0 & -H_1\\
0 & I_{m_2k} & 0\\
0 & 0 & I_{m_3k}
\end{bmatrix},
\]
where
\[
D:=I_{m_1k}-\text{diag}\!\big(H^{(1,1)}-\bb{1}'u^{(1)},\ldots,
H^{(m_1,m_1)}-\bb{1}'u^{(m_1)}\big).
\]

If $m_1\neq0$ and $m_2=0$,
\[
D_u=
\begin{bmatrix}
D & -H_1\\
0 & I_{m_3k}
\end{bmatrix},
\]
and if $m_1=0$ and $m_2\neq0$, we simply set $D_u=I_{mk}$.

\subsubsection{\textbf{\emph{Interacting scale matrix}}}
Finally, we define the \emph{interacting scale matrix} $\Gamma_u$ as the block
diagonal matrix
\begin{equation}\label{Gammau}
\Gamma_u:=\text{diag}\big(\Gamma(u^{(1)}),\Gamma(u^{(2)}),\ldots,\Gamma(u^{(m)})\big),
\end{equation}
where for a probability vector $p=(p_1,\ldots,p_k)$,
\begin{equation}\label{Gamma}
\Gamma(p):=\text{diag}(p)-p'p.
\end{equation}

\begin{remark}
The interacting reinforcement matrix $H$ governs the asymptotic behavior of the
fluctuations $U_n-u$. The exponent matrix $D_u$ and the scale matrix $\Gamma_u$
are used to describe the limiting distributions. Both $H$ and $D_u$ are upper
block triangular. In particular,
\[
Sp(H)=
\begin{cases}
\bigcup_{j=1}^{m_1} Sp(H^{(j,j)}), & m_1\neq0,\\
\{1\}, & m_1=0.
\end{cases}
\]
Moreover, if $\lambda\neq1$ is an eigenvalue of $H$, then its algebraic
multiplicity is $\nu(\lambda)$.
\end{remark}

The eigen-structure of $D_u$ is closely related to that of $H$ (see Lemma \ref{lem:3}) which leads to the following Jordan decomposition.

\subsubsection{\textbf{\emph{Jordan decomposition of the exponent matrix}}}
Since $D_u=I_{mk}$ when $m_1=0$, we henceforth assume $m_1\neq0$. Let
\[
Sp(H)=\bigcup_{j=1}^{m_1}Sp(H^{(j,j)})=\{\lambda_1,\ldots,\lambda_s\},
\qquad \lambda_1=1,
\]
and define $\nu_1:=m_1k+m_2k+m_3k$, and $\nu_j:=\nu(\lambda_j)$ for $j\ge2$.
Then, by Lemma~\ref{lem:3}, the Jordan decomposition of $D_u$ is given by
\[
T^{-1}D_uT=\text{diag}(J_1,J_2,\ldots,J_s),
\]
where $J_1$ is the $\nu_1\times\nu_1$ Jordan block associated with eigenvalue $1$,
and for $j\ge2$, $J_j$ is the $\nu_j\times\nu_j$ Jordan block associated with the
eigenvalue $1-\lambda_j$ of $D_u$.

Now we are ready to present our main theorem of this section:
\begin{theorem}[\textbf{Rate of convergence of urn proportions}]
\label{thm:rate}
In the sum-reinforcement scheme, suppose first that
$G_v=\emptyset$ (equivalently, $m_1=0$), that is, all oldest ancestors
of $v$ are stubborn. Then
\begin{equation}
    \sqrt{n}(U_n-u)
    \stackrel{d}{\longrightarrow}
    N_{mk}\bigl(0,H'\Gamma_uH\bigr)
    \qquad\text{as }n\rightarrow\infty.
\end{equation}

Suppose next that $G_v\neq\emptyset$. If $\gamma<\frac{1}{2}$, then
\begin{equation}\label{clt1}
    \sqrt{n}(U_n-u)
    \stackrel{d}{\longrightarrow}
    N_{mk}(0,\Sigma_1)
    \qquad\text{as }n\rightarrow\infty,
\end{equation}
where
\begin{equation}
    \Sigma_1
    :=
    \int_0^{\infty}
    \bigl(e^{(\frac{1}{2}I_{mk}-D_u)x}\bigr)'
    H'\Gamma_uH
    e^{(\frac{1}{2}I_{mk}-D_u)x}
    \,dx.
\end{equation}

If $\gamma=\frac{1}{2}$, then
\begin{equation}\label{clt2}
    \frac{\sqrt{n}}{(\log n)^{\nu-\frac{1}{2}}}
    (U_n-u)
    \stackrel{d}{\longrightarrow}
    N_{mk}(0,\Sigma_2)
    \qquad\text{as }n\rightarrow\infty,
\end{equation}
where
\begin{equation}
    \Sigma_2
    :=
    \lim_{n\rightarrow\infty}
    \frac{1}{(\log n)^{2\nu-1}}
    \int_0^{\log n}
    \bigl(e^{(\frac{1}{2}I_{mk}-D_u)x}\bigr)'
    H'\Gamma_uH
    e^{(\frac{1}{2}I_{mk}-D_u)x}
    \,dx.
\end{equation}

Finally, if $\frac{1}{2}<\gamma<1$, then
\begin{equation}\label{as}
    \frac{n^{1-\gamma}}{(\log n)^{\nu-1}}(U_n-u)
    -
    \sum_{\substack{l:\,\Re(\lambda_l)=\gamma,\nu_l=\nu}}
    e^{i\Im(\lambda_l)\log n}\xi_l\mathbf{r}_l
    \stackrel{a.s.}{\longrightarrow}0
    \qquad\text{as }n\rightarrow\infty,
\end{equation}
where $i=\sqrt{-1}$, and $\xi_l$'s are complex random variables defined
for all $l$ such that $\Re(\lambda_l)=\gamma$ and $\nu_l=\nu$.
Here, $\mathbf{r}_l$ is a specific left eigenvector (possibly generalized)
of $D_u$ corresponding to the eigenvalue $1-\lambda_l$. More precisely,
if $\Re(\lambda_l)=\gamma$, then $\mathbf{r}_l$ is the
$(\nu_1+\nu_2+\cdots+\nu_l)$-th row of the matrix $T^{-1}$.

Moreover, if $D_u$ is diagonalizable in $\mathbb{R}^{mk\times mk}$, then
there exists a real random vector $L$ such that
\begin{equation}
    \frac{n^{1-\gamma}}{(\log n)^{\nu-1}}(U_n-u)
    \stackrel{a.s.}{\longrightarrow}
    L
    \qquad\text{as }n\rightarrow\infty.
\end{equation}
\end{theorem}

\begin{remark}
In the average-reinforcement scheme, \autoref{thm:rate} remains valid with the exact same expressions for the limiting distributions. This is because the matrices \(H\) and \(D_u\) remain unchanged if each matrix \(R^{(l,j)}\) for \(l \rightsquigarrow j\) is substituted by \(R^{(l,j)}/d_v^{\mathrm in}\).
\end{remark}

\section{Proof of the main results}\label{sec:proofs}
The proof of strong convergence is based on the key observation that, conditional
on the entire color--composition process of a fixed urn $v$, the sequence of
draws from that urn is independent. While this may seem counterintuitive, it is a
direct consequence of the directed acyclic structure of the underlying interaction graph.
Specifically, the configuration of an urn $v$ depends only on the draws from
its ancestor urns and is completely independent of the draws from $v$ itself.
The following lemma provides a rigorous formulation of this fact.

\begin{lemma}\label{lem:1}
For every $v\in V\setminus G$ and $n\ge1$,
\[
\Q\left(\Z{j}{v}=z_j,\;1\le j\le n \,\middle|\, \U{l}{v},\,l\ge0\right)
=\prod_{j=1}^n \U{j-1}{v}(z_j)',
\]
for every choices $z_1,\ldots,z_n\in\{e_1,e_2,\ldots,e_k\}$.
\end{lemma}

\begin{proof}
It is enough to prove the lemma for sum-reinforcement scheme. Note that if $v\in S$ then the lemma follows trivially, since then $\U{j}{v}=\U{0}{v}$ for all $j\geq0$ and conditioned on $\sigma\big(\U{0}{v}\big)$, $\Z{1}{v},\Z{2}{v},....$ are i.i.d random vectors.

Now suppose $v\in V\setminus(G\cup S)$ and $\Qr$ is a regular conditional probability of $\Q$ with respect to (conditioned on) the sigma algebra 
\[
\sigma\big(\C{0}{p},R^{\e} : p\in V,\e\in\E\big).
\]
Fix integers $m\geq1,n\geq2$ and choose vectors 
\[
\big\{\z{j}{p},\z{l}{v} : p\in A_v, 1\leq j\leq m+n, 1\leq l\leq n\big\}\subseteq\big\{e_1,e_2,...,e_k\big\}.
\]
Using the conditional independence of draws given the past, we obtain
\begin{align*} &\Qr\bigg(\Z{j}{p}=\z{j}{p}, \Z{l}{v}=\z{l}{v}, 1\leq j\leq m+n,1\leq l\leq n, p\in A_v\bigg)\\ =~&\Qr\bigg(\Z{1}{p}=\z{1}{p}, p\in A_v\cup\{v\}\bigg)\\
&\times\prod\limits_{l=2}^n\Qr\bigg(\Z{l}{p}=\z{l}{p}, p\in A_v\cup\{v\}~\bigg|~\Z{j}{p}=\z{j}{p},j\leq l-1, p\in A_v\cup\{v\}\bigg)\\ &\times\prod\limits_{r=1}^m\Qr\bigg(\Z{n+r}{p}=\z{n+r}{p},~p\in A_v~\bigg|~\Z{j}{p}=\z{j}{p}, \Z{l}{v}=\z{l}{v}, j\leq n+r-1,\\[-0.25cm]
&\hspace{9.5cm} l\leq n,p\in A_v\bigg)\\ =~&\bigg(\prod\limits_{l=1}^n\prod\limits_{p\in A_v\cup\{v\}}\us{l-1}{p}(\z{l}{p})'\bigg)\bigg(\prod\limits_{r=1}^m\prod\limits_{p\in A_v}\us{n+r-1}{p}(\z{n+r}{p})'\bigg) \end{align*} where 
\[
\us{j}{p}:=\frac{\C{j}{p}}{\C{j}{p}\bb{1}'}
\]
and $\C{l}{p}$'s are recursively defined as 
\[
\C{l}{p}=\C{l-1}{p}+\sum\limits_{q:q\rightsquigarrow p}\z{l}{q}R^{(q,p)}.
\]
Similarly 
\[
\Qr\bigg(\Z{j}{p}=\z{j}{p}, p\in A_v, j\leq m+n\bigg)=\prod\limits_{j=1}^{m+n}\prod\limits_{p\in A_v}\us{j-1}{p}(\z{j}{p})'.
\]
Taking the ratio of the above expressions yields \begin{align*} &\Qr\bigg(\Z{j}{v}=\z{j}{v}, j\leq n~\bigg|~\Z{l}{p}=\z{l}{p}, l\leq m+n, p\in A_v\bigg)\\ =~&\frac{\bigg(\prod\limits_{l=1}^n\prod\limits_{p\in A_v\cup\{v\}}\us{l-1}{p}(\z{l}{p})'\bigg)\bigg(\prod\limits_{r=1}^m\prod\limits_{p\in A_v}\us{n+r-1}{p}(\z{n+l}{p})'\bigg)}{\prod\limits_{j=1}^{m+n}\prod\limits_{p\in A_v}\us{j-1}{p}(\z{j}{p})'}\\ =~&\prod\limits_{j=1}^n\us{j-1}{v}(\z{j}{v})'\\ =~&\Q\bigg(\Z{j}{v}=\z{j}{v}, j\leq n~\bigg|~\U{l}{v}=\us{l}{v}, l<m+n\bigg) \end{align*} last equality follows because 
$$\sigma\big(\U{l}{v} : l<m+n\big)\subseteq\sigma\big(\Z{l}{p},\C{0}{q},R^{\e} : l\leq m+n, p\in A_v, q\in V, \e\in\E\big)
$$ 
hence we get the following \begin{equation} \Q\bigg(\Z{j}{v}=\z{j}{v}, j\leq n~\bigg|~\U{l}{v},~ l<m+n\bigg)=\prod\limits_{j=1}^n\U{j-1}{v}(\z{j}{v})' \end{equation} letting $m\longrightarrow\infty$ and using Doob's forward convergence theorem, the lemma follows.
\end{proof}

Now that we have independence, it becomes much easier to handle the situation. The following lemma explains the strong asymptotic behavior of the color-count statistic for any non-generator urn.

\begin{lemma}\label{lem:2}
    For $v\in V\setminus G$,
    \begin{equation}
    \frac{1}{n}\bigg(\sum\limits_{j=1}^n\Z{j}{v}-\sum\limits_{j=1}^n\U{j-1}{v}\bigg)\xrightarrow[]{a.s.}0,\qquad\text{as}~n\rightarrow\infty
    \end{equation}
\end{lemma}     
\begin{proof}
Note that for $v\in S$, conditioned on $\sigma\big(\U{m}{v} : m\geq0\big)=\sigma\big(\U{0}{v}\big)$, the drawings $\Z{1}{v},\Z{2}{v},...$ are i.i.d. random vectors, so the lemma follows by strong law of large numbers.

For $v\in V\setminus G,~n\geq1$ and $\epsilon>0$,
\begin{align*}
&\Q\left(\left\lVert\sum\limits_{j=1}^n\Z{j}{v}-\sum\limits_{j=1}^n\U{j-1}{v}\right\rVert>n\epsilon\,\middle|\,~\U{m}{v}, m\geq0\right)\\
    \leq~&\sum\limits_{l=1}^k\Q\left(\left|\sum\limits_{j=1}^n\Z{j}{v}(l)-\sum\limits_{j=1}^n\U{j-1}{v}(l)\right|>\frac{n\epsilon}{\sqrt{k}}\,\middle|\,\U{m}{v}, m\geq0\right)\\
    \leq~&2k\exp{\left(-\frac{2n\epsilon^2}{k}\right)}
\end{align*}
Since by Lemma \ref{lem:1}, conditioned on $\sigma\big(\U{m}{v} : m\geq0\big)$, the random variables $\Z{1}{v}(l),\Z{2}{v}(l),...$ are independent and $\Z{j}{v}(l)$ has a Bernoulli distribution with parameter $\U{j-1}{v}(l)$, for all $j\geq1$, so that the last step follows by the Hoeffding's inequality.

Finally note that
\begin{equation}
\sum\limits_{n=1}^{\infty}\Q\bigg(\bigg\lVert\sum\limits_{i=1}^n\Z{i}{v}-\sum\limits_{i=1}^n\U{i-1}{v}\bigg\rVert>n\epsilon\bigg)~<~\infty
\end{equation}
hence the lemma follows by Borel-Cantelli lemma. 
\end{proof} 

\subsection*{Proof of \autoref{thm:strongconvergence}} We will prove the theorem for the sum-reinforcement scheme only. 

For a directed path $l=v_1v_2...v_n~(n\geq2)$ we denote the number of distinct vertices in $l$ by $|l|$. For two  vertices $u,v$ and a path $l=v_1v_2...v_n$ we write $u\xrightarrow{l}v$ if $v_1=u$ and $v_n=v$.

Now for $v\in G\cup S$ the urn-proportions $\U{n}{v}$ either remain fixed at $\U{0}{v}$ (if $v\in S$) or converge almost surely to $\U{\infty}{v}:=\N{\infty}{v}R^{(v,v)}/\N{\infty}{v}R^{(v,v)}\mathbbm{1}'$ (if $v\in G$) by Assumption \ref{asm:strongconvergence}.

Note that 
\[
V\setminus(G\cup S)=\bigcup\limits_{j=2}^{\infty}V_j
\]
where 
\[V_j:=\bigg\{v\in V:\max\{|l| : \exists p\in G\cup S~\text{with}~p\xrightarrow{l}v\}=j\bigg\}
\] 
for all $j\geq2$. Clearly $V_j\cap V_{j'}=\emptyset$, whenever $j\neq j'.$ We will use induction to prove that for all $v\in\bigcup\limits_{j=2}^{\infty}V_j$ the urn-proportions $\U{n}{v}$ will admit almost sure limits $\U{\infty}{v}$ which satisfy \eqref{lim}. First let $v\in V_2$ then clearly
\begin{equation}\frac{1}{n}\C{n}{v}=\frac{1}{n}\C{0}{v}+\sum\limits_{u:u\rightsquigarrow v,u\in G}\N{n}{u}R^{(u,v)}+\sum\limits_{u:u\rightsquigarrow v,u\in S}\N{n}{u}R^{(u,v)}
\end{equation}
since $\N{n}{u}\xrightarrow{a.s.}\N{\infty}{u}$ for all $u\in G$ (Assumption \ref{asm:strongconvergence}) and $\N{n}{u}\xrightarrow{a.s.}\U{0}{u}$ for all $u\in S$ (by strong law of large numbers) we have  
\begin{equation}\frac{1}{n}\C{n}{v}\xrightarrow{a.s.}\sum\limits_{u:u\rightsquigarrow v,u\in G}\N{\infty}{u}R^{(u,v)}+\sum\limits_{u:u\rightsquigarrow v,u\in S}\U{0}{u}R^{(u,v)}
\end{equation}
now let \begin{equation}\U{\infty}{v}:=\frac{\sum\limits_{u:u\rightsquigarrow v,u\in G}\N{\infty}{u}R^{(u,v)}+\sum\limits_{u:u\rightsquigarrow v,u\in S}\U{0}{u}R^{(u,v)}}{\sum\limits_{u:u\rightsquigarrow v,u\in G}\N{\infty}{u}R^{(u,v)}\mathbbm{1}'+\sum\limits_{u:u\rightsquigarrow v,u\in S}\U{0}{u}R^{(u,v)}\mathbbm{1}'}\end{equation}
then clearly $\U{n}{v}=(\C{n}{v}\mathbbm{1}')^{-1}\C{n}{v}\xrightarrow{a.s.}\U{\infty}{v}$. Hence \autoref{thm:strongconvergence} holds for all $v\in V_2$. Now assume that for some $m\geq2$,  for all $u\in \bigcup\limits_{j=2}^m V_j$ the urn proportions $\U{n}{u}$ converge to limit $\U{\infty}{u}$ which satisfy \eqref{lim}. Let $v\in V_{m+1}$ and write 
\begin{equation}
\frac{1}{n}\C{n}{v}=\frac{1}{n}\C{0}{v}+A_n+B_n 
\end{equation}
where 
\begin{equation}
\begin{split}A_n &=\sum\limits_{u:u\rightsquigarrow v,u\in G}\N{n}{u}R^{(u,v)}+\sum\limits_{u:u\rightsquigarrow v,u\in V\setminus G}\left(\frac{1}{n}\sum\limits_{j=1}^n\U{j-1}{u}\right)R^{(u,v)}\\
&=\sum\limits_{u:u\rightsquigarrow v,u\in G}\N{n}{u}R^{(u,v)}+\sum\limits_{u:u\rightsquigarrow v,u\in\bigcup\limits_{l=2}^m V_l}\left(\frac{1}{n}\sum\limits_{j=1}^n\U{j-1}{u}\right)R^{(u,v)}
\end{split}
\end{equation} 
and 
\begin{equation}B_n=\sum\limits_{u:u\rightsquigarrow v,u\in V\setminus G}\frac{1}{n}\sum\limits_{j=1}^n\left(\Z{j}{u}-\U{j-1}{u}\right)R^{(u,v)}.
\end{equation}
Note that by our inductive assumption and by Assumption \ref{asm:strongconvergence} we get 
\begin{equation}A_n\xrightarrow{a.s.}\sum\limits_{u:u\rightsquigarrow v,u\in G}\C{\infty}{u}R^{(u,v)}+\sum\limits_{u:u\rightsquigarrow v,u\in V\setminus G}\U{\infty}{u}R^{(u,v)}
\end{equation} also by Lemma \ref{lem:2}
\begin{equation}
B_n\xrightarrow{a.s.}0
\end{equation}
now define 
\begin{equation}
\U{\infty}{v}:=\frac{\sum\limits_{u:u\rightsquigarrow v,u\in G}\N{\infty}{u}R^{(u,v)}+\sum\limits_{u:u\rightsquigarrow v,u\in V\setminus G}\U{\infty}{u}R^{(u,v)}}{\sum\limits_{u:u\rightsquigarrow v,u\in G}\N{\infty}{u}R^{(u,v)}\mathbbm{1}'+\sum\limits_{u:u\rightsquigarrow v,u\in V\setminus G}\U{\infty}{u}R^{(u,v)}\mathbbm{1}'}
\end{equation}
then clearly $\U{n}{v}\xrightarrow{a.s.}\U{\infty}{v}$. Hence the statement of \autoref{thm:strongconvergence} holds for all $v\in\bigcup\limits_{j=2}^{m+1}V_j$ and hence by induction it holds for all $v\in\bigcup\limits_{j=2}^{\infty}V_j$ .
\qed

\subsection*{Proof of Corollary \ref{cor:synchronization}} The urn at the generator vertex $g\in G$ is a classical multicolor P\'olya urn
with initial configuration $(n_1,\ldots,n_k)$. It is well known that
$\N{n}{g}\xrightarrow{a.s.}\U{\infty}{g}$, where $\U{\infty}{g}$ follows a
Dirichlet distribution with parameter $(n_1,\ldots,n_k)$. Hence
Assumption~\ref{asm:strongconvergence} holds automatically. The recursion~\eqref{lim} then implies, by
induction, that $\U{\infty}{v}=\U{\infty}{g}$ for all $v\in V$.\qed

\subsection*{Proof of Corollary \ref{cor:deterministiclimit}} Under Assumption~\ref{asm:deterministiclimit}, each generator $g\in G$ is a generalized P\'olya urn with
balanced and irreducible replacement matrix $R^{(g,g)}$, implying
$\N{n}{g}\xrightarrow{a.s.}\us{}{g}$ (see
\cite{athreya1968embedding, gouet1997strong, gangopadhyay2022almost}). Thus
Assumption~\ref{asm:strongconvergence} holds, and the recursion~\eqref{lim} directly yields~\eqref{11}.
\qed
\\\\
We now proceed to the proof of \autoref{thm:rate}. Recall the definitions of $U_n$, $H$,
$D_u$, and $\Gamma_u$ from \autoref{subsec:rate}. The following lemma describes the
relationship between the eigen-structure of the interacting replacement matrix
$H$ and that of the exponent matrix $D_u$.

\begin{lemma}\label{lem:3}
The spectrum of the exponent matrix $D_u$ is given by
\[
Sp(D_u)=
\begin{cases}
\{1\}\cup\{1-\lambda:\lambda\in Sp(H)\setminus\{1\}\}, & \text{if } m_1\neq0,\\[0.25cm]
\{1\}, & \text{if } m_1=0.
\end{cases}
\]
Moreover, the algebraic multiplicity of the eigenvalue $1$ of $D_u$ equals
$m_1+m_2k+m_3k$, while for any eigenvalue $\mu\neq1$ of $D_u$, its algebraic
multiplicity is $\nu(1-\mu)$.
\end{lemma}

\begin{proof}
For $m_1=0$ there is nothing to prove so without loss assume that $m_1\neq0$. For each
$j\in\{1,\ldots,m_1\}$, the matrix $H^{(j,j)}$ is row-stochastic. Since $D_u$ is
block upper triangular, its characteristic polynomial is the product of the
characteristic polynomials of its diagonal blocks, which are either of the form
$I_k-H^{(j,j)}+\bb{1}'\us{}{j}$ or $I_k$. Thus, it suffices to establish the
following claim.

\smallskip
\noindent
\emph{Claim.} Let $A\in\mathbb{R}^{k\times k}$ be an irreducible row-stochastic
matrix with distinct eigenvalues $\lambda_1,\ldots,\lambda_l$, dominant left
eigenvector $a$ (normalized so that $a\bb{1}'=1$), and characteristic polynomial
\[
p_A(x)=(x-\lambda_1)^{p_1}(x-\lambda_2)^{p_2}\cdots(x-\lambda_l)^{p_l}.
\]
Then the characteristic polynomial of
\[
K:=I_k-A+\bb{1}'a
\]
is given by
\[
p_K(x)=(x-\lambda_1)^{p_1}(x-1+\lambda_2)^{p_2}\cdots(x-1+\lambda_l)^{p_l}.
\]

\smallskip
By the Perron--Frobenius theorem, we may assume $\lambda_1=1$ and $p_1=1$. Let
\[
P^{-1}AP=\text{diag}(J_1,J_2,\ldots,J_l)
\]
be the Jordan decomposition of $A$, where $J_1=1$ and $J_2,\ldots,J_l$ are the
Jordan blocks corresponding to $\lambda_2,\ldots,\lambda_l$. Since the rows of
$P^{-1}$ and the columns of $P$ are respectively the left and right (possibly
generalized) eigenvectors of $A$, and since $a$ and $\bb{1}'$ are respectively the
left and right eigenvectors associated with $\lambda_1=1$, it follows that
\[
P^{-1}\bb{1}'=e_1' \quad\text{and}\quad aP=e_1.
\]
Consequently,
\[
\begin{aligned}
P^{-1}KP
&=P^{-1}\big(I_k-A+\bb{1}'a\big)P \\
&=\text{diag}(0,I_{p_2}-J_2,\ldots,I_{p_l}-J_l)+e_1'e_1 \\
&=\text{diag}\big(1,I_{p_2}-J_2,\ldots,I_{p_l}-J_l\big).
\end{aligned}
\]
Therefore,
\[
\begin{aligned}
p_K(x)
&=\det(xI_k-K) \\
&=\det\!\big(\text{diag}(x-1,(x-1)I_{p_2}+J_2,\ldots,(x-1)I_{p_l}+J_l)\big) \\
&=(x-1)(x-1+\lambda_2)^{p_2}\cdots(x-1+\lambda_l)^{p_l},
\end{aligned}
\]
which proves the claim and hence the lemma.
\end{proof}

We are now in a position to prove \autoref{thm:rate}. The key idea is to reformulate the
process $\{U_n : n\geq1\}$ as a stochastic approximation scheme with drift function $-h$,
where the Jacobian $Dh(u)$ at the limiting point $u$ has eigenvalues with strictly
positive real parts. This allows us to apply the fluctuation results of
Theorems~2.1,2.2 and 2.3 in \cite{zhang2016central}.

\subsection*{Proof of \autoref{thm:rate}}

We first introduce a matrix-valued map
\[
\omega:\mathbb{R}^{mk}\longrightarrow\mathbb{R}^{mk\times mk}.
\]
For $x=(x_1,\ldots,x_m)\in\mathbb{R}^{mk}$, where
$x_j\in\mathbb{R}^k$ for each $j=1,\ldots,m$, define
\begin{equation}
\begin{split}
\omega(x)
&:=
\operatorname{diag}\Big(
\operatorname{diag}(x_1\bb{1}_k'\bb{1}_k),
\ldots,
\operatorname{diag}(x_m\bb{1}_k'\bb{1}_k)
\Big)\\
&=
\begin{bmatrix}
P_1 & 0 & \cdots & 0\\
0 & P_2 & \cdots & 0\\
\vdots & \vdots & \ddots & \vdots\\
0 & 0 & \cdots & P_m
\end{bmatrix},
\end{split}
\end{equation}
where
\begin{equation}
\begin{split}
P_j
&:=\operatorname{diag}(x_j\mathbbm{1}_k'\mathbbm{1}_k)\\
&=
\begin{bmatrix}
x_j\mathbbm{1}_k' & 0 & \cdots & 0\\
0 & x_j\mathbbm{1}_k' & \cdots & 0\\
\vdots & \vdots & \ddots & \vdots\\
0 & 0 & \cdots & x_j\mathbbm{1}_k'
\end{bmatrix}.
\end{split}
\end{equation}

Let $E_n$ be the $mk$-dimensional vector defined by
\begin{equation}
E_n:=
\begin{cases}
\big(
\C{n}{1},\ldots,\C{n}{m_1},
\N{n}{m_1+1},\ldots,\N{n}{m_1+m_2},\\
\qquad\qquad\qquad\qquad\quad
\C{n}{m_1+m_2+1},\ldots,\C{n}{m}
\big),
&\text{if}~ m_1\neq0,\ m_2\neq0,\\[0.25cm]
\big(
\C{n}{1},\ldots,\C{n}{m_1},
\C{n}{m_1+1},\ldots,\C{n}{m}
\big),
&\text{if}~m_1\neq0,\ m_2=0,\\[0.25cm]
\big(
\N{n}{1},\ldots,\N{n}{m_2},
\C{n}{m_2+1},\ldots,\C{n}{m}
\big),
&\text{if}~ m_1=0,\ m_2\neq0.
\end{cases}
\end{equation}
It can be checked that
\begin{equation}
E_n=U_n\omega(E_n).
\end{equation}
Moreover,
\begin{equation}\label{bal}
\omega(E_{n+1})-\omega(E_n)=B,
\end{equation}
where
\begin{equation}
B=\operatorname{diag}(b_1,\ldots,b_m),
\end{equation}
and $b_j$ denotes the total balance of vertex $j$; see
Corollary~\ref{cor:deterministiclimit}.

Next, let
\begin{equation}
Z_n:=(\Z{n}{1},\ldots,\Z{n}{m})
\end{equation}
and define the filtration
\begin{equation}
\F{n}:=\sigma(Z_j:j\leq n).
\end{equation}
By the dynamics of the sum-reinforcement scheme,
\begin{equation}
\mathbb{E}\left[Z_{n+1}\mid\F{n}\right]
=
(\U{n}{1},\ldots,\U{n}{m})
=:\hat{U}_n.
\end{equation}
Furthermore, by \eqref{2},
\begin{align}
& E_{n+1} = E_n+Z_{n+1}HB\\[0.25cm]
\Leftrightarrow~& U_{n+1}\omega(E{n+1}) = U_n\omega(E_n)+Z_{n+1}HB\\[0.25cm]
\Leftrightarrow~& (U_{n+1}-U_n)\omega(E{n+1}) = -U_nB+Z_{n+1}HB
\end{align}

We can therefore write the dynamics of $U_n$ in the form of a stochastic
approximation scheme:
\begin{equation}\label{stochapprox}
U_{n+1}-U_n
=
\frac{1}{n+1}
\big(
-h(U_n)+\Delta M_{n+1}+R_{n+1}
\big),
\qquad\forall n\geq0,
\end{equation}
where $h:\mathbb{R}^{mk}\to\mathbb{R}^{mk}$ is defined by
\begin{equation}
h(x)=x-2\bar{x}H+\bar{x}H\omega(\bar{x}),
\end{equation}
and, for
$x=(x_1,\ldots,x_m)\in\mathbb{R}^{mk}$ with
$x_j\in\mathbb{R}^k$, we set
\begin{equation}
\bar{x}
=
\begin{cases}
\big(
x_1,\ldots,x_{m_1},
\U{0}{m_1+1},\ldots,\U{0}{m_1+m_2},\\
\qquad\qquad\qquad\qquad\qquad x_{m_1+m_2+1},\ldots,x_m
\big),
&\text{if}~ m_1\neq0,\ m_2\neq0,\\[0.25cm]
\big(
\U{0}{1},\ldots,\U{0}{m_2},
x_{m_2+1},\ldots,x_m
\big),
&\text{if}~ m_1=0,\ m_2\neq0,\\[0.25cm]
x,
&\text{if}~ m_2=0.
\end{cases}
\end{equation}

The sequence $\{\Delta M_{n+1}:n\ge0\}$ is an
$\F{n+1}$-adapted martingale difference sequence given by
\begin{equation}
\Delta M_{n+1}
=
(Z_{n+1}-\hat{U}_n)H,
\end{equation}
while $R_{n+1}$ is the $\F{n+1}$-adapted error sequence defined by
\begin{equation}
R_{n+1}
=
\big(-U_n+Z_{n+1}H\big)
\big((n+1)\omega(E_{n+1})^{-1}-B^{-1}\big).
\end{equation}
Since the urn-proportions take values in a compact set, there exist
non-random positive constants $c_1,c_2$ such that
\begin{equation}
\lVert\Delta M_n\rVert\leq c_1
\qquad a.s.,
\end{equation}
and, by \eqref{bal},
\begin{equation}\label{rem}
\lVert R_n\rVert\leq\frac{c_2}{n}
\qquad a.s.
\end{equation}

We next verify the conditions concerning the martingale difference
sequence. Since, conditionally on $\F{n}$, the components of $Z_{n+1}$
are independent, its conditional covariance matrix is
\begin{equation}
\begin{split}
\operatorname{Var}(Z_{n+1}\mid\F{n})
&=
\mathbb{E}\left[
(Z_{n+1}-\hat{U}_n)'(Z_{n+1}-\hat{U}_n)
\mid\F{n}
\right]\\
&=
\operatorname{diag}\big(
\operatorname{Var}(\Z{n+1}{1}\mid\F{n}),
\ldots,
\operatorname{Var}(\Z{n+1}{m}\mid\F{n})
\big)\\
&=
\operatorname{diag}\big(
\Gamma(\U{n}{1}),\ldots,\Gamma(\U{n}{m})
\big),
\end{split}
\end{equation}
where $\Gamma(\cdot)$ is defined in \eqref{Gamma}. Since
$U_n\xrightarrow{a.s.}u$, we have
\begin{equation}
\operatorname{Var}(Z_{n+1}\mid\F{n})
\xrightarrow{a.s.}\Gamma_u,
\end{equation}
and hence
\begin{equation}\label{martingalediff}
\begin{split}
\mathbb{E}\left[
(\Delta M_{n+1})'\Delta M_{n+1}
\mid\F{n}
\right]
&=
H'\operatorname{Var}(Z_{n+1}\mid\F{n})H\\
&\xrightarrow{a.s.}
H'\Gamma_uH.
\end{split}
\end{equation}

We now focus on the function $h$. Observe that there exist
$mk\times mk$ real matrices $A_1,\ldots,A_{mk}$ such that
\begin{equation}
h(x)
=
\big(
xA_1x',xA_2x',\ldots,xA_{mk}x'
\big),
\qquad\forall x\in\mathbb{R}^{mk}.
\end{equation}
Thus, $h$ is continuously differentiable, and there exists a constant
$C>0$ such that
\begin{equation}\label{normeq}
\lVert h(x)\rVert\leq C\lVert x\rVert^2.
\end{equation}
The derivative of $h$ at $u$ is
\[
Dh(u)=2\big(A_1u',\ldots,A_{mk}u'\big).
\]
Consequently, by \eqref{normeq},
\begin{equation}\label{grad}
h(x)
=
h(u)+(x-u)Dh(u)
+O\big(\lVert x-u\rVert^2\big),
\qquad\text{as}~x\rightarrow u.
\end{equation}

For $x=(x_1,\ldots,x_m)\in\mathbb{R}^{mk}$, write
\[
h(x)=\big(h_1(x),\ldots,h_m(x)\big).
\]
Then
\begin{equation}
h_j(x)=
\begin{cases}
x_j-2x_jH^{(j,j)}
+(x_j\bb{1}')x_jH^{(j,j)},
&\text{if}~ 1\leq j\leq m_1,\\[0.25cm]
x_j-\U{0}{j},
&\text{if}~ m_1+1\leq j\leq m_1+m_2,\\[0.25cm]
x_j
-2\sum_{l=1}^{m_1}x_lH^{(l,j)}\\[0.15cm]
\qquad-2\sum_{l=m_1+1}^{m_1+m_2}\U{0}{l}H^{(l,j)}
\\[0.15cm]
\qquad
+\sum_{l=1}^{m_1}(x_j\bb{1}')x_lH^{(l,j)}
\\[0.15cm]
\qquad+\sum_{l=m_1+1}^{m_1+m_2}
(x_j\bb{1}')\U{0}{l}H^{(l,j)},
&\text{if}~ m_1+m_2+1\leq j\leq m.
\end{cases}
\end{equation}

It can be checked that the derivative of $h$ at $x$ is the $mk\times mk$ block matrix
\[
Dh(x)
=
\left(
\frac{\partial h_j}{\partial x_l}(x)
\right)_{1\leq l,j\leq m},
\]
where
\begin{equation}\label{derivative}
\frac{\partial h_j}{\partial x_l}(x)
=
\begin{cases}
I_k-2H^{(j,j)}
+\bb{1}'x_jH^{(j,j)}
+(x_j\bb{1}')H^{(j,j)},
&\text{if}~ l=j\leq m_1,\\[0.25cm]
I_k,
&\text{if}~ l=j>m_1,\\[0.25cm]
-2H^{(l,j)}
+(x_j\bb{1}')H^{(l,j)},
&\text{if}~ l\leq m_1,\\&j>m_1+m_2,\\[0.25cm]
0,
& \text{otherwise}.
\end{cases}
\end{equation}
By the definition of the exponent matrix $D_u$, it follows from
\eqref{derivative} that
\begin{equation}
Dh(u)=D_u.
\end{equation}

Moreover, by \eqref{u1}, \eqref{u2}, and \eqref{u3},
\begin{equation}
u(I_{mk}-H)=0.
\end{equation}
Therefore,
\begin{align}
h(u)
&=u-2uH+uH\omega(u)\\
&=u-uH\\
&=0.
\end{align}
Thus, $u$ is an equilibrium point of the equation $h(x)=0$.
By Lemma~\ref{lem:3}, all eigenvalues of the derivative $D_u$ at $u$
have positive real parts. Hence, Assumption~2.1 of
\cite{zhang2016central} is satisfied. Moreover, \eqref{grad} shows that
Assumption~2.2 of \cite{zhang2016central} is satisfied for every
$\epsilon\in(0,1)$.

Since the martingale differences $\Delta M_n$ are bounded, we immediately
obtain
\begin{equation}
\frac{1}{n}\sum_{j=1}^n
\mathbb{E}\left[
\lVert\Delta M_j\rVert^2
\mathbbm{1}_{\{\lVert\Delta M_j\rVert\geq\epsilon\sqrt{n}\}}
\,\middle|\,\F{j-1}
\right]
\xrightarrow{a.s.}0.
\end{equation}
Furthermore, by \eqref{martingalediff},
\begin{equation}
\frac{1}{n}\sum_{j=1}^n
\mathbb{E}\left[
(\Delta M_j)'\Delta M_j
\mid\F{j-1}
\right]
\xrightarrow{a.s.}
H'\Gamma_uH.
\end{equation}
Thus, Assumption~2.3 of \cite{zhang2016central} is also satisfied with
\[
\Gamma=H'\Gamma_uH.
\]

Finally, by \eqref{rem},
\begin{equation}\label{rem2}
\sum_{j=1}^nR_j=O(\log n)
\qquad{a.s.}
\end{equation}

Let
\begin{equation}
\rho:=\min\{\Re(\mu):\mu\in\operatorname{Sp}(D_u)\}.
\end{equation}

We now distinguish the different cases according to the value of $\rho$.

\paragraph{Case 1: $G_v=\emptyset$.}
If $G_v=\emptyset$ (equivalently, $m_1=0$), then
$D_u=I_{mk}$ and hence
\[
\rho=1>\frac12.
\]
Therefore, all the conditions of Theorem~2.3 of
\cite{zhang2016central} are satisfied. Consequently,
\begin{equation}
\sqrt{n}(U_n-u)
\stackrel{d}{\longrightarrow}
\mathcal{N}_{mk}(0,\Sigma),
\qquad \text{as}~n\rightarrow\infty,
\end{equation}
where
\begin{equation}
\begin{split}
\Sigma
&=
\int_0^\infty
\big(e^{(\frac12I_{mk}-D_u)x}\big)'
H'\Gamma_uH
e^{(\frac12I_{mk}-D_u)x}\,dx\\[0.25cm]
&=
\left(\int_0^\infty e^{-x}\,dx\right)
H'\Gamma_uH\\[0.25cm]
&=
H'\Gamma_uH.
\end{split}
\end{equation}
This proves the first assertion.

\paragraph{Case 2: $G_v\neq\emptyset$ and $\gamma<\frac12$.}
Suppose that $G_v\neq\emptyset$ (equivalently, $m_1\neq0$). Then,
by Lemma~\ref{lem:3},
\[
\rho=1-\gamma.
\]
If $\gamma<\frac12$, then $\rho>\frac12$. Hence Theorem~2.3 of
\cite{zhang2016central} applies and yields the central limit theorem
in \eqref{clt1}.

\paragraph{Case 3: $G_v\neq\emptyset$ and $\gamma=\frac12$.}
If $\gamma=\frac12$, then
\[
\rho=\frac12.
\]
In this case, all the conditions of Theorem~2.1 of
\cite{zhang2016central} are satisfied. Therefore, the central limit
theorem stated in \eqref{clt2} follows.

\paragraph{Case 4: $G_v\neq\emptyset$ and $\gamma>\frac12$.}
Finally, suppose that $\gamma>\frac12$. Then
\[
0<\rho=1-\gamma<\frac12.
\]
By \eqref{martingalediff},
\begin{equation}
\sum_{j=1}^n
\mathbb{E}\left[
(\Delta M_j)'\Delta M_j
\mid\F{j-1}
\right]
=
O(n)
\qquad{a.s.}
\end{equation}
Moreover, \eqref{rem2} implies that
\begin{equation}
\sum_{j=1}^nR_j
=
o\big(n^{1-\rho-\delta_0}\big)
\qquad{a.s.}
\end{equation}
for every $\delta_0\in(0,1-\rho)$. Hence, all the conditions of
Theorem~2.2 of \cite{zhang2016central} are satisfied. Consequently,
there exist complex random variables $\xi_1,\ldots,\xi_s$ such that
\begin{equation}\label{as2}
\frac{n^\rho}{(\log n)^{\nu-1}}(U_n-u)
-
\sum_{l\in S_u}
e^{-i\Im(1-\lambda_l)\log n}
\xi_l\mathbf{e}_lT^{-1}
\xrightarrow{a.s.}0,
\qquad n\rightarrow\infty,
\end{equation}
where $S_u$ is the collection of indices $l\in\{1,\ldots,s\}$ such that
$1-\lambda_l$ is an eigenvalue of $D_u$ with real part $\rho$ and
algebraic multiplicity $\nu_l=\nu$. Recall that
\[
{Sp}(H)=\{\lambda_1,\ldots,\lambda_s\}.
\]

By Lemma~\ref{lem:3}, $S_u$ is also the collection of indices $l$
such that $\lambda_l$ is an eigenvalue of $H$ with algebraic
multiplicity $\nu_l=\nu$ and real part $\gamma$. For $l\in S_u$, define
\[
\mathbf{r}_l:=\mathbf{e}_lT^{-1}.
\]
Here, $\mathbf{e}_l$ denotes the $mk$-dimensional row vector whose
$(\nu_1+\cdots+\nu_l)$-th coordinate is $1$ and remaining
coordinates are zero.
Recalling the Jordan decomposition of $D_u$, this gives precisely the
almost sure convergence in \eqref{as}.

It remains to identify the random variables $\xi_l$. For this, we
follow the proof of Theorem~2.2 in \cite{zhang2016central}. In their
notation, the matrix $\mathbf{H}$ corresponds to $D_u$ here. Recall
the Jordan decomposition of $D_u$. Set
\begin{equation}
(U_n-u)T
=
y_n
=
(y_{n,1},\ldots,y_{n,s}),
\end{equation}
where $y_{n,l}$ is a row vector of length $\nu_l$, and define
\begin{equation}
\widetilde{\Pi}_j^{n,l}
:=
\prod_{c=j+1}^n
\left(I_{\nu_l}-\frac{J_l}{c}\right).
\end{equation}
For $l\in S_u$, since
$\Re(1-\lambda_l)<1$,
\[
\widetilde{\Pi}_0^{n,l}n^{J_l}
\longrightarrow A_l
\]
for some invertible matrix $A_l$. The proof of Theorem~2.2 of
\cite{zhang2016central} then gives
\begin{equation}\label{xi}
\begin{split}
y_{n,l}n^{J_l}
&\xrightarrow{a.s.}
\left[
y_{0,l}
+
\sum_{j=1}^{\infty}
\frac{\widetilde{s}_{j,l}}{j}
\frac{I_{\nu_l}-J_l}{j+1}
\big(\widetilde{\Pi}_0^{j+1,l}\big)^{-1}
\right]A_l\\
&=:\xi_l,
\end{split}
\end{equation}
where
\begin{equation}
\widetilde{s}_n
=
(\widetilde{s}_{n,1},\ldots,\widetilde{s}_{n,s})
:=
\left(\sum_{i=1}^nR_i^*\right)T
\end{equation}
and
\begin{equation}
R_{n+1}^*
:=
(n+1)(U_{n+1}-U_n)
+
(U_n-u)D_u.
\end{equation}

If $D_u$ is diagonalizable in $\mathbb{R}^{mk\times mk}$, then $T$ and
the matrices $J_l$ are all real. Hence, by \eqref{xi}, $\xi_l$ is real
for every $l\in S_u$. Since
\[
\Im(1-\lambda_l)=0,
\qquad\forall l\in S_u,
\]
it follows from \eqref{as2} that
\begin{equation}
\frac{n^\rho}{(\log n)^{\nu-1}}(U_n-u)
\xrightarrow{a.s.}
\sum_{l\in S_u}
\xi_l\mathbf{e}_lT^{-1}
=:L.
\end{equation}
Thus, $L$ is a real $mk$-dimensional random vector. This completes the
proof.\qed

\bibliographystyle{plain}
\bibliography{Ref}

\end{document}